\documentclass[a4paper, 12pt]{amsart}

\usepackage[a4paper, margin=2cm]{geometry}

\usepackage{amsmath}
\usepackage{amssymb}
\usepackage{calc}
\usepackage{enumitem}
\usepackage{graphicx}
\usepackage{tikz-cd}
\usepackage{url}
\usepackage{xcolor}
\usepackage{etoolbox}
\usepackage{tabularx}
\usepackage{mathtools}
\usepackage{xcolor}
\usepackage{mathrsfs}
\usepackage{hyperref}

\hypersetup{
  colorlinks   = true, %Colours links instead of ugly boxes
  urlcolor     = {blue!50!black}, %Colour for external hyperlinks
  linkcolor    = {blue!50!black}, %Colour of internal links
  citecolor   = {red!50!black} %Colour of citations
}

\numberwithin{equation}{section}

\usepackage{cleveref}

\newtheorem{customthm}{Theorem}
\newenvironment{numthm}[1]
{\renewcommand\thecustomthm{#1}\customthm}
{\endcustomthm}

\theoremstyle{plain}
\newtheorem{thm}{Theorem}[section]

\newtheorem{prop}[thm]{Proposition}
\newtheorem{lem}[thm]{Lemma}
\newtheorem*{lem*}{Lemma}
\newtheorem*{cor*}{Corollary}
\theoremstyle{definition}
\newtheorem{defn}[thm]{Definition}

\newtheorem{notn}[thm]{Notation}
\newtheorem*{qn*}{Question}

\theoremstyle{remark}
\newtheorem{rem}[thm]{Remark}
\newtheorem*{term*}{Terminology}

\DeclareMathOperator{\Age}{Age}

\DeclareMathOperator{\Aut}{Aut}

\DeclareMathOperator{\id}{id}
\DeclareMathOperator{\Nbhd}{N}

\DeclareMathOperator{\St}{St}

\DeclareMathOperator{\tp}{tp}

\newcommand{\N}{\mathbb{N}}
\newcommand{\Q}{\mathbb{Q}}

\newcommand{\ou}{\,\vee\,}
\newcommand{\ic}{\mathbin{\bot}}

\newcommand{\mc}[1]{\mathcal{#1}}

\newcommand{\all}{\forall\,}
\newcommand{\ex}{\exists\,}
\newcommand{\sub}{\subseteq}

\newcommand{\emp}{\varnothing}

\newcommand{\fin}{\subseteq_{\text{fin\!}}}

\newcommand{\Ainf}{\Age_{\text{inf}}}
\newcommand{\op}{{\,\text{op}}}

\def\ind{\mathrel{\raise0.2ex\hbox{\ooalign{\hidewidth$\vert$\hidewidth\cr\raise-0.9ex\hbox{$\smile$}}}}}

\newenvironment{subproof}[1][\proofname]{%
\par\textit{Proof of claim. }}{%
  \hfill $\blacksquare$ \par 
}

\newcommand{\Fr}{Fra\"{i}ss\'{e} }
\newcommand{\Frn}{Fra\"{i}ss\'{e}}

\newcommand{\Ka}{Kat\v{e}tov }

\allowdisplaybreaks

\newcommand{\lessorinc}{%
  \mathrel{%
    \vcenter{\offinterlineskip
      \ialign{##\cr$<$\cr\noalign{\kern-2pt}$\ic$\cr}%
    }%
  }%
}

\usepackage{stackengine}
\stackMath

\makeatletter
\def\author@andify{%
  \nxandlist {\unskip ,\penalty-1 \space\ignorespaces}%
    {\unskip {} \@@and~}%
    {\unskip \penalty-2 \space \@@and~}%
}
\makeatother

\author{Aleksandra Kwiatkowska}
\address{\parbox{\linewidth}{Aleksandra Kwiatkowska\\
Institut f\"{u}r Mathematische Logik und Grundlagenforschung, Universit\"{a}t  M\"{u}nster\\ Einsteinstraße 62\\
48149 M\"{u}nster, Germany\\
and \\
Instytut Matematyczny, Uniwersytet Wroc{\l}awski\\
pl.\ Grunwaldzki 2/4 \\ 
50-384 Wroc{\l}aw, Poland
}
}
\email{aleksandra.kwiatkowska@math.uni.wroc.pl}

\author{Rob Sullivan}
\address{\parbox{\linewidth}{Rob Sullivan\\
Institute of Computer Science, Czech Academy of Sciences\\
Pod Vodárenskou věží 271/2\\
182 00 Prague, Czech Republic
}}
\email{robertsullivan1990+maths@gmail.com}

\author{Jeroen Winkel}
\address{Jeroen Winkel\\
Independent Researcher}
\email{winkeljeroen+maths@gmail.com}



\title{Embeddings into the generic poset}

\date{\today}

\begin{document}

\begin{abstract}
    Let $M$ be the generic poset, defined as the \Fr limit of the class of finite posets. We show that every countably infinite poset $A$ can be embedded with coinfinite image into $M$ so that each automorphism of the image of $A$ extends uniquely to an automorphism of $M$.
\end{abstract}

\maketitle

\section{Introduction}

We say that a first-order structure $M$ is a \emph{\Fr structure} if it is countably infinite and ultrahomogeneous (see \cite{Hod93}, \cite{Mac11} for background). Let $M$ be a \Fr structure, and let $\Age_\omega(M)$ be the class of structures $A$ which are embeddable in $M$, including infinite structures. For finite $A \in \Age_\omega(M)$, immediately by the definition of ultrahomogeneity we have that for each embedding $f : A \to M$, each automorphism $g$ of $f(A)$ extends to an automorphism $\theta(g)$ of $M$, giving a map $\theta : \Aut(f(A)) \to \Aut(M)$. Two questions immediately arise:
\begin{itemize}
    \item If $A$ is infinite, does there exist an embedding $f : A \to M$ with such an extension map $\theta$?
    \item Can we additionally require that $\theta$ is a group embedding (i.e.\ preserves composition), or that $\theta$ is unique?
\end{itemize}

Let $A \in \Age_\omega(M)$, and let $f : A \to M$ be an embedding. We say that $f$ is \emph{group-extensive} if there exists a group embedding $\theta : \Aut(f(A)) \to \Aut(M)$ such that $\theta(g)$ extends $g$ for all $g \in \Aut(f(A))$. We call $f$ \emph{uniquely extensive} if each automorphism of $f(A)$ extends uniquely to an automorphism of $M$. Any uniquely extensive embedding is group-extensive (by uniqueness, the extension map must preserve composition), and it is straightforward to see that an embedding $f : A \to M$ is uniquely extensive if and only if it is group-extensive and the pointwise-stabiliser of $f(A)$ in $\Aut(M)$ is trivial. If $M$ has strong amalgamation, then it is easy to see that no finite $A$ admits a uniquely extensive embedding, so for uniquely extensive embeddings we restrict attention to infinite $A$.

We will say that $\Age_\omega(M)$ is \emph{group-extensible} if each $A \in \Age_\omega(M)$ admits a group-extensive embedding into $M$. We write $\Ainf(M)$ for the set of infinite elements of $\Age_\omega(M)$, and say that $\Ainf(M)$ is \emph{uniquely extensible} if each of its elements admits a uniquely extensive embedding into $M$.

Bilge and Melleray showed that $\Age_\omega(M)$ is group-extensible for structures $M$ with free amalgamation (\cite[Theorem 3.9]{BM13} -- an adaptation of an argument of Uspenskij in \cite{Usp90}), and this was generalised to structures $M$ with a stationary independence relation (SIR) by M\"{u}ller (\cite{Mul16}). (SIRs were originally defined in \cite{TZ13}: they encode the notion of a symmetric canonical amalgam, one particular case of which is the free amalgam.) The construction in each case is similar, and is sometimes referred to as a \emph{\Ka tower construction} (\Ka used this to produce the Urysohn space in \cite{Kat88}). A general abstract categorical formulation of this construction is given in \cite{KM17}, and investigated further in \cite{KSW27}.

Henson (\cite{Hen71}) showed that $\Ainf(M)$ is uniquely extensible for $M$ equal to the random graph (also shown independently by Macpherson and Woodrow in \cite{MW92}), and Bilge (\cite{Bil12}) extended this result in his PhD thesis to any transitive free amalgamation structure $M$ in a finite relational language (and satisfying the non-triviality condition that there are two distinct points related by some relation, i.e.\ $M$ is not an indiscernible set). For the particular case of the generic $K_n$-free graph, see \cite{BJ12}. In \cite[Proposition 6.1]{KSW27}, the authors of the present paper showed that $\Ainf(M)$ is uniquely extensible for $M$ equal to the generic ordered graph, again using a \Ka tower construction; we believe that a straightforward generalisation of this, analogous to \cite[Theorem 8.30]{Bil12}, should give unique extensibility for the generic order expansion of a non-trivial transitive free amalgamation structure in a finite relational language.

Given M\"{u}ller's result that $\Age_\omega(M)$ is group-extensible for $M$ with a SIR, it is therefore natural to wonder if the presence of a SIR on $M$ implies that $\Ainf(M)$ is uniquely extensible. The graph $M$ consisting of the disjoint union of $\omega$-many countably infinite cliques shows that this is not the case, but one may argue that this example is somehow too trivial. Jaligot (\cite{Jal07}) showed that $\Ainf(M)$ is uniquely extensible for $M$ equal to the random tournament (which admits an independence relation satisfying all the axioms of a SIR except symmetry), and Melleray (\cite{Mel05}) showed that, for any separable metric space $A$, there exists an isometric embedding $f$ of $A$ into the (complete) Urysohn space $\mathbb{U}$ such that every isometry of $f(A)$ extends uniquely to an isometry of $\mathbb{U}$. (Note that the rational Urysohn space admits a (local) SIR: see \cite[Example 2.2.1(a)]{TZ13}.)

In this paper, we consider the case of the \emph{generic poset} $M$: the \Fr limit of the class of all finite posets. Observe that $M$ admits a stationary independence relation: for $A, B, C \fin M$, define $B \ind_A C$ if the poset $A \cup B \cup C$ is the transitive closure of the free amalgam of $B \cup A$, $C \cup A$ over $A$. More explicitly: we define $B \ind_A C$ if for $b \in B \setminus A$, $c \in C \setminus A$ we have:
\begin{itemize}
    \item $b \neq c$;
    \item $b < c$ $\iff$ there exists $a \in A$ with $b < a < c$;
    \item $b > c$ $\iff$ there exists $a \in A$ with $b > a > c$.
\end{itemize}

The main (and indeed only) theorem of this paper states that $\Ainf(M)$ is uniquely extensible:

\begin{numthm}{\ref{main thm}}
    Let $M$ denote the generic poset, and let $A$ be a countably infinite poset. Then there exists an embedding $f : A \to M$ with coinfinite image such that each automorphism of $f(A)$ extends uniquely to an automorphism of $M$.
\end{numthm}

We also note that the induced group embedding $\Aut(f(A)) \to \Aut(M)$ given by the above theorem will in addition be an embedding of topological groups with closed image (this is immediate by \Cref{ac props}\ref{ac:stab}). Also, the fact that the embedding $f$ given by \Cref{main thm} has coinfinite image ensures that we exclude the trivial case of surjective $f$ when $A \cong M$.

\subsection*{Comparison with previous results}

Our construction of the embedding $f$ differs significantly from the previous examples mentioned above and is very particular to the generic poset; essentially, the fundamental difficulty that prevents a straightforward generalisation of the previous results of Bilge for free amalgamation classes or Jaligot for tournaments is that the poset relation $<$ is transitive.

We now briefly outline the proof ideas of previous results.

In the case of linear orders (which one might initially expect to be analogous to partial orders), we have that $\Ainf(\Q, <)$ is not uniquely extensible: we show that any $A \in \Ainf(\Q, <)$ with $A \not\cong \Q$ does not admit a uniquely extensive embedding into $\Q$ as follows. Let $f : A \to \Q$ be an embedding. As $f(A) \not\cong \Q$ we have that $f(A)$ is either not dense as a linear order or has an endpoint, and so there are $u, v \in \Q$, $u < v$, with $(u, v) \cap f(A) = \emp$. Let $g$ be a non-trivial automorphism of $(u, v)$. Then $\id_\Q$ and $g \cup \id_{\Q \setminus (u, v)}$ are distinct automorphisms of $\Q$ extending $\id_{f(A)}$, so $f$ is not uniquely extensive. (In fact, a stronger statement is true: if a countably infinite linear order $A$ admits a uniquely extensive embedding $f : A \to \Q$, then $A \cong \Q$ and $f$ is surjective. See \cite[Proposition 8.33]{Bil12}.)

A sketch proof of unique extensibility for the random graph (\cite{Hen71}) is as follows: given a countably infinite graph $A$, inductively construct a chain $A = M_0 \sub M_1 \sub \cdots$ where, to build $M_k$ from $M_{k - 1}$, for each $F \fin M_{k - 1}$ with $|F \cap M_0| = k$ we add a new vertex $v_F$ adjacent to exactly $F$. Note that, for $k \geq 1$, each $g \in \Aut(M_{k-1})$ fixing $M_0$ setwise extends uniquely to an element of $\Aut(M_k)$ via $v_F \mapsto v_{g(F)}$. Let $M_\omega = \bigcup_{k < \omega} M_k$; then each $g \in \Aut(A)$ extends to an element of $\Aut(M_\omega)$. For $v \in M_\omega$ let $\Nbhd(v)$ denote the neighbourhood of $v$ in $M_\omega$. It is clear that any $h \in \Aut(M_\omega)$ extending $g \in \Aut(A)$ satisfies $|\Nbhd(h(v)) \cap M_0| = |\Nbhd(v) \cap M_0|$ for each $v \in M_\omega \setminus M_0$, and therefore stabilises each ``layer" $M_k$ setwise, giving uniqueness of automorphism extension. To show that $M_\omega$ is isomorphic to the random graph, one checks that for each pair of finite disjoint sets $U, V \sub M_\omega$, there is a ``witness" $m \in M_\omega \setminus (U \cup V)$ with $m$ adjacent to all of $U$ and none of $V$: for this, as we have $U, V \sub M_{k-1}$ for some $k \geq |U|$, we may pad out $U$ with vertices from $M_0 \setminus V$ to produce $F \sub M_{k-1}$ with $|F \cap M_0| = k$, and the new vertex of $M_k$ whose $M_{k-1}$-neighbourhood is $F$ then gives the required witness $m$.

An analogous construction (\cite{BJ12}) by Bilge and Jaligot for the generic triangle-free graph again uses the invariant $|\Nbhd(v) \cap M_0|$ to define layers $M_k$ which are setwise-stabilised by any extension of an automorphism of $M_0$, but there is an additional complication: one cannot begin with $M_0 = A$, as then it is not always possible to pad out $U$ with vertices from $M_0$ as in the argument above. (For example, consider the case where $A$ is an infinite star graph with central vertex $a$. Let $U = \{a, u\}$ be an anticlique with $u \in M_1 \setminus M_0$. Any padding-out $F$ of $U$ with an additional vertex from $A$ in an attempt to find a witness for $U$ in $M_2$ will not work, as there cannot be a new vertex of $M_2$ adjacent to $F$: this would create a triangle.) Bilge and Jaligot fix this issue by taking $M_0 = A \sqcup \N$, where the edges of $\N$ are given by $\{n, n + 1\}$ for $n \in \N$ (so the graph on $\N$ is rigid, i.e.\ has trivial automorphism group), and where there are no edges between $A$ and $\N$. They then perform the same construction as above, always requiring with each new vertex added that it is adjacent to at least one element of $A$. This ensures that any $h \in \Aut(M_\omega)$ extending an element of $\Aut(A)$ must fix $\N$, and therefore $h$ setwise-stabilises each $M_k$. Padding out with the new $M_0$ now works: one uses elements of $\N$.

In the case of the generic poset that we are concerned with in this paper, we also use the idea of attaching a poset to $A$ to define $M_0$ and then inductively constructing a chain $M_0 \sub M_1 \sub \cdots$. However, our approach requires several new ideas (also in the construction of $M_0$): it is difficult to see how the construction of Bilge and Jaligot could work, as it seems that the transitivity of $<$ breaks natural invariants that one might try.

\subsection*{Structure of the paper and general proof outline.}

Let $M$ denote the generic poset, and let $A$ be a countable poset. Let $G = \Aut(A)$. The rough overview of the construction of a uniquely extensive embedding $A \to M$ is as follows. First, we attach a new poset $R \cup S \cup T$ to $A$ to produce a poset $M_0$ (see \Cref{d: M_0}). Each point of $R$ and $S$ will have its type over $A$ being a $G$-invariant limit point in the type space consisting of types of one-point extensions over $A$ (see \Cref{d: ext types}), which we discuss in \Cref{s: valid triples} and \Cref{s: limits}, and this will enable us to easily extend the action $G \curvearrowright A$ to $G \curvearrowright A \cup R \cup S$. The set $T$ will also be defined relative to $A \cup R \cup S$ in such a way that we may easily extend the action $G \curvearrowright A \cup R \cup S$ to $G \curvearrowright M_0$.

In \Cref{ss: acceptable chains} we then inductively construct a chain $M_0 \sub M_1 \sub \cdots$ of $G$-posets (posets with a $G$-action by automorphisms), and we do this in such a way that the union $P$ is isomorphic to the generic poset (by guaranteeing witnesses for each possible one-point extension of a finite substructure -- see \Cref{enough aps}). As we construct $P$ as a $G$-poset, we immediately have that each automorphism of $A$ extends to an automorphism of $P$. To guarantee uniqueness of automorphism extension -- the most difficult and technically intricate part of the construction -- we also ensure that $P$ has the property that any $h \in \Aut(P)$ with $h|_A = \id_A$ must also be the identity on $M_0$ (see \Cref{lem:minimality} and the proof of the main \Cref{main thm}), and furthermore is the identity on the upward-closure $S^+$ of $S$ (see \Cref{ss: stabilisers} and \Cref{s: moieties}). This is then sufficient to show that in fact $h = \id_P$ (see \Cref{uniq on +-closed}), giving uniqueness.

\subsection*{Future directions} Given the technical nature of the construction of uniquely extensive embeddings for the generic poset in this paper, we imagine that more general results for a broader class of structures may require different techniques. We would also be keen to know if the main theorem of this paper leads to other interesting results regarding the automorphism group of the generic poset: certainly, this group is not as well-understood as the automorphism groups of other well-known \Fr structures.

\section{Valid triples and exterior types} \label{s: valid triples}

We first define some terminology and notation that we will use throughout the rest of this paper.

\begin{defn} \label{+ and - closure}
    Let $P$ be a poset, and let $Q \sub P$. We write
    \begin{align*}
        Q^- &= \{v \in P \mid \ex w \in Q, v \leq w\},\\
        Q^+ &= \{v \in P \mid \ex w \in Q, v \geq w\}.
    \end{align*}
    We call $Q^-$ and $Q^+$ the \emph{downward-} and \emph{upward-closure} of $Q$ respectively.
    
    For $v \in P$, we write $v^- = \{v\}^-$, $v^+ = \{v\}^+$. For $Q \sub P$ and $v \in P$, we write $v \ic Q$ if $v \ic w$ for all $w \in Q$.
\end{defn}

\begin{defn} \label{d: ext types}
    Let $A$ be a poset. We say that a poset $E$ is a \emph{one-point extension} of $A$ if $A \sub E$ and $|E \setminus A| = 1$, and we refer to the unique element $e$ of $E \setminus A$ as the \emph{extension point} of $E$. We write $\tp(e / A)$ for the \emph{quantifier-free type of $e$ over $A$}. All types considered in this paper will be quantifier-free.
    
    We define $\eta(A)$ to be the set of types $\tp(e/A)$ where $e$ is the extension point of a one-point extension of $A$, and call $\eta(A)$ the set of \emph{exterior types over $A$}. We will usually write $p \in \eta(A)$ rather than $p(x) \in \eta(A)$, with the free variable $x$ being clear from context. For $A' \sub A$ we write $p|_{A'}$ for the restriction of $p$ to $A'$ (that is, the set of formulae in $p$ only involving elements of $A'$).
\end{defn}
\begin{rem}
    We call the elements of $\eta(A)$ \emph{exterior} types rather than external types, to avoid a terminological clash with other notions of external type in model theory.
\end{rem}

\begin{defn}
    Let $A$ be a poset. For $p \in \eta(A)$, if $\phi(x, \bar{a}) \in p$, we will usually conflate $x$ and $p$ for the sake of brevity and just write $\phi(p, \bar{a})$. For example, rather than writing that the formula $x < a$ is contained in $p$, we simply write $p < a$. (Informally, we think of $p$ as a one-point extension taken up to isomorphism over $A$.)

    For $p \in \eta(A)$, we let \[U_p = \{a \in A \mid a < p\},\quad V_p = \{a \in A \mid a \ic p\},\quad W_p = \{a \in A \mid a > p\}.\]

    Let $(U, V, W)$ be a partition of $A$. We say that $(U, V, W)$ is a \emph{valid triple} for $A$ if there exists $p \in \eta(A)$ with $(U_p, V_p, W_p) = (U, V, W)$. It is immediate that there is a one-to-one correspondence between elements of $\eta(A)$ and valid triples given by $p \mapsto (U_p, V_p, W_p)$. In the sequel we will mildly abuse notation and identify exterior types with valid triples when convenient: we thus often write $p = (U_p, V_p, W_p)$. (One could instead define the set of exterior types entirely in terms of valid triples if desired.)

    If $A$ is finite and $(U, V, W)$ is a valid triple for $A$, then we say that $(U, V, W)$ is a \emph{finite valid triple}.

    Let $P$ be a poset. We refer to a valid triple for a substructure of $P$ as a \emph{valid triple in $P$}.
\end{defn}

The proof of the following lemma is straightforward.

\begin{lem}
    Let $A$ be a poset, and let $(U, V, W)$ be a partition of $A$. Then $(U, V, W)$ is a valid triple if and only if the following hold:
    \begin{itemize}
        \item $u < w$ for $u \in U$, $w \in W$;
        \item $\neg(u > v)$ for $u \in U$, $v \in V$;
        \item $\neg(w < v)$ for $w \in W$, $v \in V$.
    \end{itemize}
\end{lem}

\begin{defn}
    Let $A, B$ be posets with $A \sub B$. Let $(U, V, W)$ be a valid triple for $A$. We call $b \in B \setminus A$ a \emph{witness} for $(U, V, W)$ if $\tp(b/A) = (U, V, W)$.
\end{defn}

Recall that in the introduction we defined the \emph{generic poset} $M$ to be the \Fr limit of the class of finite posets. See \cite[Chapter 7]{Hod93}, \cite[Section 2.2]{Mac11} for more background: to amalgamate finite posets $B, C$ over $A$, one takes the free amalgam and then the transitive closure. We assume that the reader is familiar with the basic theory of \Fr limits. In particular, we shall often use the following:
\begin{enumerate}[label=(\roman*)]
    \item the generic poset $M$ contains a witness for each finite valid triple in $M$ (this follows by the extension property of $M$);
    \item if $P$ is a countable poset containing a witness for each finite valid triple in $P$, then $P$ is isomorphic to the generic poset $M$ (this is by \Frn's theorem). 
\end{enumerate}

\textbf{We now fix a countably infinite poset $A$ for the remainder of this paper.}

\begin{defn} \label{ext types topology}
    We equip $\eta(A)$ with the following topology: for $A_0 \fin A$ and for $p_0$ the type over $A_0$ of a one-point extension of $A_0$, we specify that $\langle p_0 \rangle = \{p \in \eta(A) \mid p \supseteq p_0\}$ is a basic open set of $\eta(A)$.

    Let $(U_0, V_0, W_0) = p_0$. We will also often mildly abuse notation and write $\langle U_0, V_0, W_0 \rangle$ for the basic open set $\langle p_0 \rangle$.
\end{defn}

\begin{lem}
    $\eta(A)$ is a Polish space.
\end{lem}
\begin{proof}
    The topology given in \Cref{ext types topology} is second countable, as $A$ is countable and the language $\{<\}$ of posets is finite. To define a metric on $\eta(A)$, we enumerate $A$ as $a_0, a_1, \cdots$. For $k \in \N$, let $A_k = \{a_i \mid i < k\}$. Given distinct $p, q \in \eta(A)$, let $m \in \N$ be the least index such that $p|_{A_m} \neq q|_{A_m}$. We define $d(p, q) = \frac{1}{m+1}$. It is straightforward to see that $d$ is a metric compatible with the topology on $\eta(A)$, and that $\eta(A)$ is (sequentially) compact and hence complete with respect to $d$.
\end{proof}

\begin{rem}
    As the theory of the generic poset is $\omega$-categorical and has quantifier elimination, we have that $\eta(A)$ is a subspace of $S_1(A)$, the Stone space of $1$-types in the usual model-theoretic sense (where we consider $A$ as a subset of the generic poset $M$). It is straightforward to see that $\eta(A)$ is a closed subspace of $S_1(A)$: if a type is not exterior, this is witnessed by a formula $x = a$.
\end{rem}

\begin{defn}
    Let $p, q \in \eta(A)$. We say that $(p, q)$ is \emph{$<$-valid} if there exists a poset $E = A \cup \{b, c\}$ extending $A$ with $\tp(b/A) = p$, $\tp(c/A) = q$ and $b < c$. We define \emph{$\ic$-valid} and \emph{$>$-valid} similarly.
\end{defn}

The proof of the following lemma is again straightforward.

\begin{lem} \label{validity check}
    Let $p, q \in \eta(A)$. Then:
    \begin{itemize}
        \item $(p, q)$ is $<$-valid $\iff$ ($U_p \sub U_q$ and $V_p \sub U_q \cup V_q$);
        \item $(p, q)$ is $\ic$-valid $\iff$ ($U_p \cap W_q = \varnothing$ and $U_q \cap W_p = \varnothing$).
    \end{itemize}
\end{lem}

(Note that $(p, p)$ is $<$-valid and $\ic$-valid for all $p \in \eta(A)$.)

\begin{defn} \label{validity poset}
    We define a binary relation $\ll$ on $\eta(A)$ by: $p \ll q$ if $p \neq q$ and $(p, q)$ is $<$-valid. \Cref{validity check} implies that $\ll$ is a strict partial order (that is, transitive and irreflexive). 
\end{defn}

\begin{lem}\label{lem:join-and-meet}
    Let $p, q \in \eta(A)$. Then $p$ and $q$ have a meet $p \wedge q$ and a join $p \vee q$ in the strict partial order $\ll$. Also, for each basic open set $\langle p_0 \rangle$ and $p, q \in \langle p_0 \rangle$, we have that $p \wedge q, p \vee q \in \langle p_0 \rangle$.
\end{lem}
\begin{proof}
    It is straightforward to check that 
    $(U_p \cap U_q, (V_p \cap V_q) \cup (V_p \cap U_q) \cup (V_q \cap U_p), W_p \cup W_q)$ is in $\eta(A)$ and equal to the meet $p \wedge q$, and that $(U_p \cup U_q, (V_p \cap V_q) \cup (V_p \cap W_q) \cup (V_q \cap W_p), W_p \cap W_q)$ is in $\eta(A)$ and equal to the join $p \vee q$. Let $\langle U_0, V_0, W_0 \rangle$ be a basic open set containing $p$ and $q$. Then $U_0 \sub U_p \cap U_q$, $V_0 \sub V_p \cap V_q$, $W_0 \sub W_p \cap W_q$. So $\langle U_0, V_0, W_0 \rangle$ also contains $p \wedge q$ and $p \vee q$.
\end{proof}

\section{Upper and lower limits} \label{s: limits}

\begin{defn}
    Let $\theta \sub \eta(A)$ and $p \in \eta(A)$. We say that $p$ is an \emph{upper limit} of $\theta$ if for every open set $O \sub \eta(A)$ with $p \in O$, there exists $r \in \theta \cap O$ with $r \ll p$: equivalently, the type $p$ is a limit point of $\{r \in \theta \mid r \ll p\}$. We define lower limits similarly, replacing $\ll$ with $\gg$.
\end{defn}

\begin{lem}\label{lem:upper-or-lower-limit}
    Let $\theta \sub \eta(A)$ be topologically closed and also closed under (finite) meets. Then an element of $\eta(A)$ is a limit point of $\theta$ if and only if it is an upper or lower limit of $\theta$.
\end{lem}
\begin{proof}
    $\Leftarrow$: immediate. $\Rightarrow$: Suppose $p \in \eta(A)$ is a limit point of $\theta$ but is not an upper or lower limit of $\theta$. Then there are open sets $O_0$, $O_1$ containing $p$ such that, in the partial order $\ll$, the set $O_0 \cap \theta$ does not contain any types smaller than $p$ and $O_1 \cap \theta$ does not contain any types larger than $p$. Let $O = O_0 \cap O_1$. So $O \cap \theta$ only contains $p$ and types incomparable to $p$. Let $O' \subseteq O$ be a basic open set containing $p$. Then $O'$ is closed under joins and meets (see Lemma \ref{lem:join-and-meet}). As $p$ is a limit point, there is some $q \in O' \cap \theta$, $q \neq p$. Then $p \wedge q \in O' \cap \theta$, so $p \wedge q = p$. This implies $p \ll q$, contradiction.
\end{proof}

\begin{defn} \label{lambda and mu}
    Let
    \begin{align*}
        \lambda(A) &= \{p \in \eta(A) \mid p < a \,\vee\, p \ic a \text{ for all } a \in A\},\\
        \mu(A) &= \{p \in \eta(A) \mid p > a \,\vee\, p \ic a \text{ for all } a \in A\}.
    \end{align*}

    Observe that $\lambda(A)$, $\mu(A)$ are closed subsets of $\eta(A)$ (it is easy to see that their complements are open) and are closed under meets and joins, as can be seen from the explicit valid triples for the meet and the join given in the proof of \Cref{lem:join-and-meet}. 
    
    Also note that for each partition $(U, V, W)$ of $A$, we have:
    \[(U, V, W) \in \lambda(A) \iff (U = \emp \text{ and } V^- = V) \iff (U = \emp \text{ and } W^+ = W).\] We use these equivalences in the proof of \Cref{not upper limit not lower limit}.
\end{defn}

\begin{lem} \label{not upper limit not lower limit}
    Let $p \in \lambda(A)$.
    \begin{enumerate}[label=(\roman*)]
        \item \label{not upper} There is $V_0 \fin A$ with $V_0^-=V_p$ if and only if $p$ is not an upper limit of $\lambda(A)$.
        \item \label{not lower} There is $W_0 \fin A$ with $W_0^+=W_p$ if and only if $p$ is not a lower limit of $\lambda(A)$.
    \end{enumerate}
\end{lem}
\begin{proof} \hfill
    \begin{enumerate}[label=(\roman*)] \setlength{\parindent}{0pt}
        \item $\Rightarrow$: As $p \in \lambda(A)$, we have $p = (\emp, V_p, W_p)$. Let $O = \langle \emp, V_0, \emp \rangle$. We have $p \in O$ as $V_0 \sub V_p$. Suppose there exists $q \in O \cap \lambda(A)$ with $q \ll p$. As $q \ll p$, by \Cref{validity check} we have $U_q=\emp$ and $V_q \sub V_p$. As $q \in \lambda(A)$ we have $V_q^- = V_q^{}$, so as $q\in O$ we have $V_q \supseteq V_0^- = V_p$, thus $V_p = V_q$ and so $p = q$, contradiction. So $p$ is not an upper limit of $\lambda(A)$.
        
        $\Leftarrow$: Let $O = \langle \emp, V_0, W_0 \rangle$ be an open set containing $p$ such that there is no $q \in O \cap \lambda(A)$ with $q \ll p$. Let $q = (\emp, V_0^-, A\setminus V_0^-)$. Then $q\in O\cap \lambda(A)$ and $(q,p)$ is $<$-valid. So $q=p$ and thus $V_0^-=V_p$.
        \item $\Rightarrow$: Let $O = \langle\emp, \emp, W_0\rangle$. We have $p \in O$ as $W_0 \sub W_p$. Suppose there exists $q \in O \cap \lambda(A)$ with $q \gg p$. As $q \gg p$ we have $W_q \sub W_p$, and $W_q \supseteq W_0^+ = W_p$, so $W_q=W_p$. But $U_q=\emp$ so $p=q$, contradiction. Thus $p$ is not a lower limit in $\lambda(A)$. 
        
        $\Leftarrow$: Let $O=\langle \emp, V_0, W_0 \rangle$ be an open set containing $p$ such that there is no $q\in O\cap \lambda(A)$ with $q\gg p$. Let $q = (\emp, A\setminus W_0^+, W_0^+)$. Then $q\in O\cap \lambda(A)$ and $(q,p)$ is $>$-valid. So we have $q=p$ and $W_0^+=W_p$. \qedhere 
    \end{enumerate}
\end{proof}

\begin{defn} \label{action on type space}
    Let $G$ be a group, and let $G \curvearrowright A$ be an action via automorphisms. We induce an action $G \curvearrowright \eta(A)$ by translation of parameters: for $p = (U_p, V_p, W_p) \in \eta(A)$, we define $g \cdot p = (g \cdot U_p,\, g \cdot V_p,\, g \cdot W_p)$. It is immediate that $(g \cdot U_p,\, g \cdot V_p,\, g \cdot W_p)$ is a valid triple. (It is also straightforward to see that the action $G \curvearrowright \eta(A)$ is continuous, though we shall not use this.)
\end{defn}

\begin{defn} \label{shift maps}
    Let $\sigma^+ \colon \lambda(A) \to \mu(A)$ denote the \emph{shift map} $\sigma^+(\emp, V, W) = (V, W, \emp)$, and let $\sigma^-$ be the inverse of $\sigma^+$.
    
    It is easy to see that $\sigma^+$ and $\sigma^-$ are $\Aut(A)$-equivariant homeomorphisms which preserve the partial order $\ll$. It is also straightforward to see using \Cref{validity check} that for $p \in \lambda(A)$, we have that $(p, \sigma^+(p))$ is $<$-valid and $\ic$-valid.
\end{defn}
    
\begin{prop} \label{A-fixed upper limit}
    There is $p \in \lambda(A)$ such that the following hold:
    \begin{itemize}
        \item $p$ is an upper limit or a lower limit of $\lambda(A)$;
        \item $p$ is a fixed point of the action $\Aut(A) \curvearrowright \eta(A)$.
    \end{itemize}
\end{prop}

\begin{proof}
    For $a \in A$, we write $c_a \in \mathbb N \cup \{\infty\}$ for the supremum of the lengths of the descending chains in $A$ starting at $a$. We split the proof into two cases: $\sup_{a \in A} c_a = \infty$ or $\sup_{a \in A} c_a < \infty$. (Note that in the first case $A$ contains arbitrarily long finite chains, but it is possible for $A$ not to contain an infinite chain.)

    Suppose $\sup_{a \in A} c_a = \infty$. Let $V = \{a \in A \mid c_a < \infty\}$ and $W = \{a \in A \mid c_a = \infty\}$. Then for $v \in V$ there is no $w \in W$ with $w < v$, as otherwise $v$ would begin arbitrarily long finite descending chains. So $(\emp, V, W)$ is a valid triple: denote the corresponding type by $p$. We have $p \in \lambda(A)$, and it is clear that $p$ is $\Aut(A)$-fixed. Suppose for a contradiction that $p$ is neither an upper nor a lower limit of $\lambda(A)$. Then by \Cref{not upper limit not lower limit} there are $V_0 \fin V$, $W_0 \fin W$ with $V_0^- = V$, $W_0^+ = W$. Letting $n = \max(\{0\} \cup \{c_{v_0} \mid v_0 \in V_0\})$, we have $c_v \leq n$ for all $v\in V$. By assumption $\sup_{a \in A} c_a = \infty$, so $W\neq \emp$. Let $w_0 \in W_0$ be a minimal element of $W_0$ and consider a descending chain of length $n+2$ starting at $w_0$. Let $a \in A$ be the second element of this chain. Since $w_0 > a$ and $w_0$ was minimal in $W_0$, we have $a \in V$, but $c_a \geq n+1$, contradiction.

    We now consider the case $\sup_{a \in A} c_a<\infty$. There is $n \in \N$ with $\{a \in A \mid c_a = n\}$ infinite. (Note that there may be multiple such $n$ -- we choose one.) Let $V = \{a \in A \mid c_a \leq n\}$ and $W = A \setminus V$. Then $(\emp, V, W)$ is a valid triple; denote its corresponding type by $p$. It is clear that $p \in \lambda(A)$ and that $p$ is $\Aut(A)$-fixed. Then $\{a \in A \mid c_a = n\} \sub V$ is an infinite set of pairwise incomparable elements, and so $V$ does not contain a finite cofinal set. So there is no $V_0 \fin A$ with $V_0^- = V$, and thus by \Cref{not upper limit not lower limit}\ref{not upper} we have that $p$ is an upper limit of $\lambda(A)$. 
\end{proof}

\subsection{Opposite posets}

\begin{defn}
    Let $P$ be a poset. We define the poset $P^\op$ to have the same domain as $P$, and define $<^{P^\op}$ by: $u <^{P^\op} v \Leftrightarrow u >^P v$.

    Let $p \in \eta(A)$. Then we define $p^\op \in \eta(A^\op)$ by $p^\op = (W_p, V_p, U_p)$. (It is easy to check that this is a valid triple for $A^\op$.) It is also straightforward to see that $p \mapsto p^\op$ is a homeomorphism $\eta(A) \to \eta(A^\op)$, equivariant under the action of $\Aut(A) = \Aut(A^\op)$.

    It is likewise straightforward to check that, for $p, q \in \eta(A)$, we have that $(p, q)$ is $<$-valid over $A$ if and only if $(p^\op, q^\op)$ is $>$-valid over $A^\op$. So, defining the poset relation $\ll$ on $\eta(A^\op)$ exactly as for $\eta(A)$ (see \Cref{validity poset}), we have that $p^\op \ll q^\op$ in $\eta(A^\op)$ if and only if $p \gg q$ in $\eta(A)$.
\end{defn}

We will use the below lemma in our main construction in \Cref{s: main constr} to ensure that we need only consider upper limits.

\begin{lem} \label{upper limit from lower}
    Let $p$ be an $\Aut(A)$-fixed lower limit of $\lambda(A)$. Then $p^\op$ is an $\Aut(A^\op)$-fixed upper limit of $\mu(A^\op)$.
\end{lem}

The proof of the above lemma is immediate.

\section{Sets of moieties} \label{s: moieties}

In the proof of our main theorem, \Cref{main thm}, we will attach a poset $R \cup S \cup T$ to our original countable poset $A$, where $S$ is an infinite antichain. Our construction will use particular collections $\Sigma$, $\Sigma'$ of subsets of $S$ satisfying certain useful properties. In this section we show the existence of these collections of subsets, before beginning the main proof.

\begin{defn}
    Let $S$ be a countably infinite set, and let $S' \sub S$. We say that $S'$ is a \emph{moiety}\footnote{A term due to P.\ Neumann; cf.\ French \emph{moiti\'{e}.}} of $S$ if $S'$ and $S \setminus S'$ are infinite.
\end{defn}

Now we prove the main lemma of this section, which is key to the main construction in the next section.

\begin{lem}\label{build Sigmas}
    Let $S$ be a countably infinite set. There are infinite collections $\Sigma, \Sigma' \sub \mc{P}(S)$ of moieties of $S$ satisfying the following:
    \begin{enumerate}[label=(\roman*)]
        \item\label{ZinSigma} for $\mc{U}, \mc{W} \fin \Sigma$, $\mc{V} \fin \Sigma'$ and $C, D \fin S$ such that 
        \[\textstyle
        C \cup \bigcup \mc{U} \sub \bigcap \mc{W} \text{ and } (C \cup \bigcup \mc{U}) \cap (D \cup \bigcup \mc{V}) = \emp,
        \] 
        there is $Z \in \Sigma$ satisfying 
        \[\textstyle    
        C \cup \bigcup \mc{U} \sub Z \sub \bigcap \mc{W} \text{ and } Z \cap (D \cup \bigcup \mc{V}) = \emp.
        \]
        If $\mc{U} \cap \mc{W} = \emp$, there are infinitely many such $Z$.
        \item\label{ZinSigma'} for $\mc{U}, \mc{W} \fin \Sigma'$, $\mc{V} \fin \Sigma$ and $C, D \fin S$ such that 
        \[\textstyle
        C \cup \bigcup \mc{W} \sub \bigcap \mc{U} \text{ and } (C \cup \bigcup \mc{W}) \cap (D \cup \bigcup \mc{V}) = \emp, 
        \]
        there is $Z \in \Sigma'$ satisfying 
        \[\textstyle
        C \cup \bigcup \mc{W} \sub Z \sub \bigcap \mc{U} \text{ and } Z \cap (D \cup \bigcup \mc{V}) = \emp.
        \]
        If $\mc{U} \cap \mc{W} = \emp$, there are infinitely many such $Z$.
    \end{enumerate}
    (Here, we follow the usual empty intersection convention: for $\mc{T} \sub \mc{P}(S)$, we define $\bigcap \mc{T} = S$ if $\mc{T} = \emp$.)
\end{lem}
\begin{rem}
    For the convenience of the reader, we note that statement \ref{ZinSigma'} is just statement \ref{ZinSigma} with $\Sigma$, $\Sigma'$ and $\mc{U}$, $\mc{W}$ swapped.

    To see an example of how the empty intersection convention applies: taking $\mc{U} = \mc{V} = \mc{W} = \emp$, statement \ref{ZinSigma} gives that for all disjoint $C, D \fin S$, there are infinitely many $Z \in \Sigma$ with $C \sub Z$ and $Z \cap D = \varnothing$, and statement \ref{ZinSigma'} gives the same but for $\Sigma'$.
\end{rem}

\begin{proof}
    Let $\mc{L}_{\mc{PO}} = \{<\}$ be the language of posets, and let $\mc{L} = \mc{L}_{\mc{PO}} \cup \{\chi_0, \chi_1, \chi_2\}$ with $\chi_i$ a unary relation symbol for $i < 3$. For an $\mc{L}$-structure $N$, we write $N^{}_i = \chi_i^{N}$, and we say that $N$ satisfies condition $(\ast)$ if the following hold:
    \begin{itemize}
        \item $<^N$ is a partial order and $\{N_0, N_1, N_2\}$ is a partition of $N$;
        \item for $n_0 \in N_0$, $n_1 \in N_1$, $n_2 \in N_2$, the following three conditions hold:
        \begin{align*}
            n_0 &< n_1 \ou n_0 \ic n_1,\\
            n_0 &< n_2 \ou n_0 \ic n_2,\\
            n_1 &< n_2 \ou n_1 \ic n_2;
        \end{align*}
        \item for distinct $n^{}_1, n'_1 \in N_1$, we have $n^{}_1 \ic n'_1$.
    \end{itemize}

    Let $\mc{K}$ be the class of finite $\mc{L}$-structures satisfying conditions $(\ast)$. It is straightforward to check that $\mc{K}$ is an amalgamation class (take the usual poset amalgam given by the transitive closure of the free amalgam); we denote its \Fr limit by $N$. The structure $N$ satisfies condition $(\ast)$. In particular, we have that $N_1$ is an infinite antichain. We may thus identify $N_1$ with $S$, and then it is an easy check that $N_0 = S^- \setminus S$ and $N_2 = S^+ \setminus S$. (Recall \Cref{+ and - closure}.)

    It is straightforward to check (using the extension property of $N$) that $N$ satisfies the following for all distinct $u, v \in N$:
    \begin{itemize}
        \item if $u, v \in N_0$, then $u^+ \cap S \neq v^+ \cap S$ and we have $(v^+ \cap S \sub u^+ \cap S) \Rightarrow u < v$;
        \item if $u, v \in N_2$, then $u^- \cap S \neq v^- \cap S$ and we have $(u^- \cap S \sub v^- \cap S) \Rightarrow u < v$;
        \item if $u \in N_0$, $v \in N_2$ and $u < v$, then $u^+ \cap v^- \cap S \neq \emp$.
    \end{itemize}

    Let $\Sigma = \{v^- \cap S \mid v \in N_2\}$ and $\Sigma' = \{v^+ \cap S \mid v \in N_0\}$. It is easy to see that $\Sigma$, $\Sigma'$ are infinite sets of moieties of $S$.

    We now prove part (i). If $\mc{U} \cap \mc{W} \neq \emp$, then $C \cup \bigcup \mc{U} = \bigcap \mc{W}$, so take $Z = C \cup \bigcup \mc{U}$.

    Now suppose $\mc{U} \cap \mc{W} = \emp$. By the definition of $\Sigma, \Sigma'$, there are $U, W \fin N_2$ and $V \fin N_0$ such that $\mc{U} = \{ u^- \cap S \mid u \in U\}$, $\mc{V} = \{ v^+ \cap S \mid v \in V\}$, $\mc{W} = \{ w^- \cap S \mid w \in W\}$. We have $U \cap W = \emp$.

    We now show that $(U, V, W)$ is a valid triple. As $\bigcup \mc{U} \sub \bigcap \mc{W}$, for $u \in U$, $w\in W$ we have $u^- \cap S \sub w^- \cap S$, so $u < w$. As $\bigcup \mc{U} \cap \bigcup \mc{V} = \emp$, for $u \in U$, $v \in V$ we have $u^- \cap v^+ \cap S = \emp$ and so $\neg(u > v)$. For $v \in V$, $w \in W$, as $v \in S^- \setminus S$, $w \in S^+ \setminus S$ and $S$ is an antichain, we also have $\neg(v > w)$. So $(U, V, W)$ is a valid triple of a finite subset of $N$.
    
    Moreover, since $(C\cup \bigcup \mc{U}) \cap (D \cup \bigcup \mc{V}) = \emp$ and $C \sub \bigcap \mc{W}$, we also have that $(C\cup U, D\cup V, W)$ is a finite valid triple. There are infinitely many witnesses $z \in N_2$ for this finite valid triple (note that any witness in $N_2$ will satisfy $(\ast)$ as $W \sub N_2$), giving infinitely many $Z = z^- \cap S$ satisfying $C \cup \bigcup \mc{U} \sub Z \sub \bigcap \mc{W}$ and $Z \cap (D \cup \bigcup \mc{V}) = \emp$ as required.
    
    The proof of (ii) is analogous to that of (i): just swap $\mc{U}$ and $\mc{W}$ and flip the signs.
\end{proof}

\section{The main construction} \label{s: main constr}

In this section, we will prove the main result of this paper:

\begin{thm} \label{main thm}
    Let $M$ denote the generic poset, and let $A$ be a countably infinite poset. Then there exists an embedding $f : A \to M$ with coinfinite image such that each automorphism of $f(A)$ extends uniquely to an automorphism of $M$.
\end{thm}

In the terminology of the introduction, the above theorem states that $\Ainf(M)$ is uniquely extensible for $M$ equal to the generic poset. We first show in the below lemma that to prove the main theorem it suffices to restrict attention to a smaller class of posets.

\begin{lem}
    Let $\mc{U}$ be the class of countably infinite posets $A$ such that $\lambda(A)$ contains an $\Aut(A)$-fixed upper limit of $\lambda(A)$. If every element of $\mc{U}$ admits a uniquely extensive embedding into $M$, then $\Ainf(M)$ is uniquely extensible.
\end{lem}
\begin{proof}
    Let $A \in \Ainf(M) \setminus \mc{U}$. Then by \Cref{A-fixed upper limit} there exists an $\Aut(A)$-fixed lower limit $p$ of the set of types $\lambda(A)$. By \Cref{upper limit from lower} we have that $p^\op$ is an upper limit of $\mu(A^\op)$, so $\sigma^-(p^\op)$ is an upper limit of $\lambda(A^\op)$. So by assumption there is a uniquely extensive embedding $f : A^\op \to M$. But $M \cong M^\op$, so $A$ also admits a uniquely extensive embedding.
\end{proof}

\begin{rem}
    Not all countably infinite posets belong to $\mc{U}$, as we now show. Let $A = (\N, <)$. For each $p \in \mu(A)$, if $p \ic n$ for some $n \in \N$, then $p \ic m$ for all $m \geq n$, so $p$ is an isolated point of $\mu(A)$. Thus the only non-isolated point of $\mu(A)$ is $p = (A, \emp, \emp)$, which is $\ll$-maximal in $\mu(A)$. So $\mu(A)$ does not contain any lower limit of itself, and hence $\lambda(A^\op)$ does not contain any upper limit of itself: thus $(\N, <)^\op \notin \mc{U}$.
\end{rem}

\textbf{We will therefore assume from now on that for our fixed countable poset $A$, there is an $\Aut(A)$-fixed upper limit in $\lambda(A)$. We fix one such upper limit and denote it by $p$.} We let $q = \sigma^+(p)$. Then $q$ is an $\Aut(A)$-fixed upper limit of $\mu(A)$, and $(p, q)$ is $<$-valid and $\ic$-valid. (See \Cref{shift maps}.)

Our construction will begin by attaching a countable poset $R \cup S \cup T$ to $A$. We write $G = \Aut(A)$ for the remainder of this section. A \emph{$G$-poset} is a poset together with a $G$-action via automorphisms, and a \emph{$G$-poset embedding} of $G$-posets is a poset embedding which is $G$-equivariant.

\begin{defn} \label{d: M_0}
    Let $R = \{r_i \mid i < \omega\}$, $S = \{s_j \mid j < \omega\}$, $T = \{t_a \mid a \in V_p\}$ be antichains, with $R$ and $S$ countably infinite and the domains of $A, R, S, T$ disjoint. 
    
    We define a countable poset $M_0 = A \cup R \cup S \cup T$ by specifying the partial order relation between each of the posets $A, R, S, T$:
    \begin{itemize}
        \item each element of $R$ has type $p$ over $A$;
        \item each element of $S$ has type $q$ over $A$;
        \item $r_i < s_j$ if $i \geq j$ and $r_i \ic s_j$ if $i < j$;
        \item $t_a > v$ for $v \in V_p$ with $a \geq v$, and $t_a$ is incomparable to all other points of $M_0$.
    \end{itemize}

    We briefly check that $M_0$ as specified above is indeed a poset: we need to check that $<$ as defined is transitive and irreflexive. By the definition of the relation $<$ involving points of $T$, we have that $T$ is an antichain consisting of points greater than elements of $M_0 \setminus T$, so cannot create any $<$-cycles or violate transitivity. We therefore just need to check that $A \cup R \cup S$ as defined above is a poset. The relation defined on $R \cup S$ is easily seen to be a poset. For each $r \in R$, $s \in S$, $a \in A$, recalling that $\tp(r/A) = p$, $\tp(s/A) = q$: as $p = (\emp, V_p, W_p)$ and $q = (V_p, W_p, \emp)$, we cannot have $r > a$ or $s < a$, and we have $(r < a) \Rightarrow (\all r' \in R, s' \in S)(r' < a \wedge s' \ic a)$ and $(s > a) \Rightarrow (\all s' \in S, r' \in R)(s' > a \wedge r' \ic a)$. So $<$ is transitive and irreflexive for $M_0$ as required.

    Note that $R \cup S $ is a rigid poset (i.e.\ it has trivial automorphism group). Each element of $R$ is less than some element of $S$ (but is only less than finitely many elements of $S$), and each element of $T$ is not related to any element of $S$.
    
    We define a $G$-action on $M_0$ extending $G \curvearrowright A$ as follows: for each $g \in G$, $g$ fixes $R \cup S$ pointwise, and $g \cdot t_a = t_{ga}$ for $t_a \in T$. (Note that as $p, q$ are $G$-fixed, this does indeed give a $G$-action.) \textbf{From now on, we consider $M_0$ as a $G$-poset with this action.} Note that as $G \curvearrowright A$ is faithful, any $G$-poset embedding of this action will also be faithful.
\end{defn}

In the remainder of the paper, we will construct a $G$-poset $P$ extending $G \curvearrowright M_0$ such that $P$ is isomorphic to the generic poset. As $P$ will be a $G$-poset, we will immediately have that each automorphism of $A$ extends to an automorphism of $P$ -- but the key difficulty lies in guaranteeing uniqueness of automorphism extension, or equivalently that $\id_P$ is the only automorphism of $P$ extending $\id_A$. The following lemma will be crucial; it shows that, in the generic poset, the only automorphism fixing a non-empty upward-closed set pointwise is the identity.

\begin{lem} \label{uniq on +-closed}
    Let $M$ be the generic poset, and let $v \in M$. Let $f \in \Aut(M)$ fix $v^+$ pointwise. Then $f = \id_M$.
\end{lem}
\begin{proof}
    We first claim that for $m \in M$ such that $m \ic v$, we have $f(m) = m$. Let $m \in M$ with $m \ic v$. As $f(v) = v$ we have $f(m) \ic v$. Suppose for a contradiction that $f(m) \neq m$. If $f(m) > m$ or $f(m) \ic m$, then $(\{v, m\}, \{f(m)\}, \emp)$ is a valid triple, and as $M$ is the generic poset this valid triple has a witness $n \in M$. As $n \in v^+$ we have $f(n) = n$. But $m < n$ and $f(m) \ic n$, contradiction. If $f(m) < m$, then $(\{v, f(m)\}, \{m\}, \emp)$ is a valid triple and we obtain a contradiction similarly, completing the proof of our claim.

    Now let $m \in M$ with $m < v$. We have $f(m) < v$. Suppose for a contradiction that $f(m) \neq m$. If $f(m) > m$ or $f(m) \ic m$, then $(\{m\},\{f(m), v\}, \emp)$ is a valid triple with a witness in $M$, fixed by $f$ by the previous claim -- contradiction. If $f(m) < m$, we similarly use the valid triple $(\{f(m)\}, \{m, v\}, \emp)$.

    As for each $m \in M$ we have $m \geq v$, $m \ic v$ or $m < v$, we have therefore shown that $f$ fixes each point of $M$.
\end{proof} 

\subsection{Outline of the remainder of the construction} \label{ss: proof outline}

To construct the $G$-poset $P$ extending $G \curvearrowright M_0$, we will inductively construct a chain of $G$-posets $M_0 \sub M_1 \sub \cdots$, and then take $P = \bigcup_{k < \omega} M_k$. We will do this in such a way that $P$ is isomorphic to the generic poset (equivalently, $P$ contains a witness for each finite valid triple in each $M_k$), but so that any automorphism $f \in \Aut(P)$ extending $\id_A$ fixes $S^+$ pointwise: these two requirements are somewhat in tension, and this leads to the complexity of our construction.

To build $M_k$ from $M_{k - 1}$, we take any sufficiently ``nice" (we say \emph{acceptable}) type over $M_{k - 1}$ whose support outside $S$ is finite and whose downward intersection with $S$ is in a particular sense well-behaved (\Cref{def:acceptable-pair}), and then we add a witness to $M_k$ for each element of the $G$-orbit of this type (\Cref{def:acceptable-construction}), extending the $G$-action to these new witness vertices via the $G$-action on types so that the new vertices added at stage $k$ form a $G$-orbit. We call the result an \emph{acceptable chain}. 

Let $P$ be the union of an acceptable chain. By \Cref{ac props}\ref{ac:s+diff}, for $m, n \in P$ added at different stages and both lying in $S^+$, we will have $m^- \cap S \neq n^- \cap S$. So any $h \in \Aut(P)$ which fixes $S$ pointwise cannot send $m$ to $n$, and thus $h$ will setwise-stabilise each layer $M_k \setminus M_{k-1}$ of new points which lie in $S^+$. This is a halfway step to the goal of ensuring that we fix $S^+$ pointwise, but we still have several problems to solve.
\begin{itemize}
    \item We need to guarantee that any $h \in \Aut(P)$ fixing $A$ pointwise does in fact fix $S$ pointwise: see \Cref{lem:minimality} for an initial half-step here towards at least setwise-fixing $S$. In the proof of \Cref{main thm} we then take an acceptable chain satisfying an additional condition (condition \ref{fin aps not S+ witnessed}) which will indeed finally ensure that $S$ is stabilised setwise. We will also ensure that $R \cup T$ is stabilised setwise, and the rigidity of $R \cup S$ will then imply that $S$ is stabilised pointwise (this is why we introduce $R$ in the definition of $M_0$: we use it to add rigidity and thereby fix each point of $S$).
    \item We also need to take a particular acceptable chain so that its union $P$ is isomorphic to the generic poset. We show that this is possible in \Cref{enough aps}, \Cref{enough aps2} and the proof of the main \Cref{main thm} (see condition \ref{fin triples witnessed} in the proof). The need to satisfy the one-point extension property of the generic poset while still having \Cref{ac props}\ref{ac:s+diff} is the reason for introducing the collections $\Sigma, \Sigma'$ of moieties in \Cref{s: moieties} (see also properties (AC3), (AC4) in \Cref{def:acceptable-pair}, the definition of an acceptable pair).
    \item The final issue is that we need to ensure that each layer $M_k \setminus M_{k-1}$ of new points lying in $S^+$ is actually fixed pointwise by $h$, not just setwise. This is the reason for the introduction of $T$ in the definition of $M_0$. See \Cref{enough aps2} and condition \ref{nice stabilisers} in the proof of the main \Cref{main thm}.
\end{itemize}

Of course, there are several other technical points of the construction which we cannot cover in this overview -- but we hope that the above description at least gives the reader an idea of the macro-structure of the proof. The reader may also find it helpful to read \Cref{s: main constr} non-linearly: once a particular concept has been introduced, it may be helpful to skip forward to the proof of \Cref{main thm} to see how it is used in the final construction.

\subsection{Acceptable chains} \label{ss: acceptable chains}

From now on, we take $\Sigma, \Sigma'$ to be fixed infinite sets of moieties of $S$ satisfying conditions \ref{ZinSigma} and \ref{ZinSigma'} of \Cref{build Sigmas} (the lemma guarantees that such exist).

Recall the definition of the $G$-poset $M_0$ from \Cref{d: M_0}: in particular, recall that each element of $G$ fixes $S$ pointwise.

\begin{defn}\label{def:acceptable-pair}
    Let $M_0 \sub B$, where $B$ is a countable $G$-poset extending $M_0$. We say that a pair $(U,W)$ of subsets of $B$ is \emph{acceptable} if it satisfies the following conditions:
    \begin{enumerate}
        \item[(AC1)] $U < W$ (that is, $u < w$ for all $u \in U$, $w \in W$);
        \item[(AC2)] $U \setminus S$, $W \setminus S$ are finite;
        \item[(AC3)] if $U^- \cap S \neq \emp$, then we have $U^- \cap S \in \Sigma$ and also $U^- \cap S \neq m^- \cap S$ for every $m \in B$;
        \item[(AC4)] if $U^- \cap (R \cup T)=\emp$, then $W^+ \cap S \in \Sigma' \cup \{S\}$.
    \end{enumerate}
    
    For an acceptable pair $(U,W)$, we define $\tau_B(U,W)$ (or simply $\tau(U,W)$ if clear from context) to be the type over $B$ given by the valid triple $(U^-, B \setminus (U^-\cup W^+), W^+)$. 
    
    (We think of $\tau_B(U, W)$ as being ``generated" via $U < \tau_B(U, W) < W$ from the ``support" $(U, W)$, making $\tau_B(U, W)$ incomparable to all elements of $B$ for which this is possible.) 

    Note that if $(U, W)$ is an acceptable pair, then so is $(gU, gW)$ for $g \in G$, and $g \cdot \tau_B(U, W) = \tau_B(gU, gW)$. Also note that the $G$-orbit of $\tau_B(U, W)$ is countable, by (AC2) and the fact that each element of $G$ fixes $S$ pointwise.
\end{defn}

We will use the following notation for stabilisers throughout the rest of Section \ref{s: main constr}:

\begin{notn}
    Let $P$ be a $G$-poset. Considering the action $G \curvearrowright P$, for $Q$ a subset of $P$, we write $\St(Q)$ for the pointwise-stabiliser of $Q$ in $G$. Considering the action on the type space $G \curvearrowright \eta(P)$ (see \Cref{action on type space}), for $r \in \eta(P)$ we write $\St(r)$ for the stabiliser in $G$ of the type $r$. (The actions in Section \ref{s: main constr} are all group actions of $G$, and so we may write $\St$ instead of $\St_G$ without any confusion.) 
\end{notn}

\begin{defn}\label{def:acceptable-construction}
    Let $B$ be a countable $G$-poset extending $M_0$, and let $(U,W)$ be an acceptable pair of subsets of $B$. Let $\Omega$ be the $G$-orbit of $\tau(U, W)$ in $\eta(B)$ (this orbit is countable, as noted in \Cref{def:acceptable-pair}). Let $B \ast (U, W)$ denote the $G$-poset constructed by extending $B$ as follows: 
    \begin{itemize}
        \item add a point $e_r$ for each type $r \in \Omega$;
        \item to define the partial order on $B \ast (U,W)$, give each $e_r$ type $r$ over $B$, and take the transitive closure of the free amalgam over $B$ of the posets $B \cup \{e_r\}$ for $r \in \Omega$ (that is, $e_r < e_{r'}$ if and only if there exists $m\in B$ with  $e_r < m < e_{r'}$);
        \item the $G$-action is defined by $g \cdot e_r = e_{g \cdot r}$ for $g\in G$.
    \end{itemize}
    Note that the new points $\{e_r \mid r \in \Omega\}$ form a $G$-orbit, and for each new point $e_r$ we have $\St(e_r) = \St(r)$.
    
    We say that a $G$-poset $B'$ extending $B$ is an \emph{acceptable extension} if there is an acceptable pair $(U,W)$ of subsets of $B$ with $B' = B \ast (U, W)$. We call an \emph{acceptable chain} any sequence of $G$-posets $(M_k)_{k < l}$, $l \in \N \cup \{\omega\}$, such that $M_k$ is an acceptable extension of $M_{k-1}$ for each $k \geq 1$.
\end{defn}

We now show some properties satisfied by any acceptable chain.

\begin{lem}[Properties of acceptable chains] \label{ac props}
    Let $(M_k)_{k < l}$ be an acceptable chain. Let $k < l$ and let $m \in M_k$. The following properties hold:
    \begin{enumerate}[label=(\roman*)]
        \item\label{ac:stab} there is $A_0\fin A$ with $\St(A_0) \sub \St(m)$;
        \item\label{ac:s-} if $m \in S^- \setminus S$, there is $C \fin V_p$ with $m^- \cap V_p = C^- \cap V_p$;
        \item\label{ac:notrt+} if $m \notin (R \cup T)^+$ then $m^+ \cap S \in \Sigma' \cup \{S\}$;
        \item\label{ac:nots} if $m \notin S$ then $m^- \cap S \in \Sigma \cup \{\emp\}$;
        \item\label{ac:s+orb} if $m \in S^+$, then for distinct $n \in G \cdot m$ we have $m \ic n$;
        \item\label{ac:s+diff} if $m \in S^+$, then for $n \in M_k \setminus (G \cdot m)$ we have $m^- \cap S \neq n^- \cap S$.
    \end{enumerate}
    We also have that:
    \begin{itemize}
        \item[$(\ast)$] for each $k < l$, the set $\{m^- \cap S \mid m \in M_k \setminus S\}$ is finite.
    \end{itemize}
\end{lem}

\begin{proof}
    We prove each claim separately via induction on $k$. For the first five claims, when proving the induction step from $k-1$ to $k$, we may assume that $m \in M_k \setminus M_{k-1}$ and that there is some acceptable pair $(U, W)$ of subsets of $M_{k-1}$ with $m = e_{\tau(U, W)}$. 

    \ref{ac:stab}: In the base case, if $m \in A$ take $A_0 = \{m\}$, if $m = t_a \in T$ take $A_0 = \{a\}$ and if $m \in R \cup S$ take $A_0 = \emp$. Now we do the induction step. For each $n \in (U \cup W) \setminus S$, by the induction hypothesis there is $A_n \fin A$ satisfying $\St(A_n) \sub \St(n)$. Let $A_0 = \bigcup_{n \in (U \cup W) \setminus S} A_n$. By (AC2) we have $A_0 \fin A$. Then for each $g \in \St(A_0)$ and $n \in U \cup W$, we have $gn=n$ (recall that $G$ fixes $S$ pointwise), so $g \cdot \tau(U,W)=\tau(U,W)$ and hence $gm = m$.
    
    \ref{ac:s-}: In the base case we have $m \in V_p$ or $m \in R$: take $C = \{m\}$ or $C = \emp$ respectively. For the induction step, as $m = e_{\tau(U, W)} \in S^- \setminus S$ and $S$ is an antichain, we have $U \sub S^- \setminus S$, so $U$ is finite by (AC2). By the induction hypothesis, for each $u \in U$ there is $C_u \fin V_p$ with $u^- \cap V_p = C_u^- \cap V_p$. Let $C = \bigcup_{u\in U} C_u$. Then $C \fin V_p$ and $m^- \cap V_p = U^- \cap V_p = C^- \cap V_p$. 
    
    \ref{ac:notrt+}: In the base case $m \in V_p$ and $m^+ \cap S = S$. For the induction step, we have $m^+ \cap S = W^+ \cap S$, and $W^+ \cap S \in \Sigma' \cup \{S\}$ by (AC4).
    
    \ref{ac:nots}: In the base case $m^- \cap S = \emp$. For the induction step $m^- \cap S = U^- \cap S$, and $U^- \cap S \in \Sigma \cup \{\emp\}$ by (AC3).
    
    \ref{ac:s+orb}: In the base case, we have $m\in S$, so $G \cdot m = \{m\}$. We now do the induction step. Recall that $M_k \setminus M_{k-1} = G \cdot m$. As $m \in S^+$ and $G$ fixes $S$ pointwise, we have $G \cdot m \sub S^+$. Suppose for a contradiction there is $n \in G \cdot m$ with $n \neq m$ such that $m, n$ are comparable. We have $n = gm$ for some $g \in G$. If $m < n$, then $g^{-1}m < m$, so it suffices to consider the case $n < m$. By the definition of $M_k$, there is $n' \in M_{k-1}$ with $n < n' < m$. As $n \in S^+$, we have $n' \in S^+$. We have $gn' < gm = n$, so $gn' < n'$, contradicting the induction hypothesis for $n' \in M_{k-1}$.
    
    \ref{ac:s+diff}: In the base case, we have $m \in S$, so $m^- \cap S = \{m\}$ and $G \cdot m = \{m\}$. For $n \in S \setminus \{m\}$ we have $n^- \cap S = \{n\}$ and for $n \in M_0 \setminus S$ we have $n^- \cap S = \emp$. Now we do the induction step. If $n \notin S^+$ then $n^- \cap S = \emp$, and so as $m \in S^+$ we immediately have $m^- \cap S \neq n^- \cap S$. So we may assume $m, n \in S^+$. If both $m, n \in M_{k-1}$, we are done by the induction assumption, so without loss of generality we may assume $m \in M_k \setminus M_{k-1}$ and $n \in M_{k-1}$ (as $m, n$ are in different $G$-orbits). As before we write $m = e_{\tau(U, W)}$. By (AC3) we have $m^- \cap S = U^- \cap S \neq n^- \cap S$.
    
    $(\ast)$: In the base case, $m^- \cap S = \emp$ for $m \in M_0 \setminus S$. For the induction step, observe that as $G$ fixes $S$ pointwise, each element of $M_k \setminus M_{k-1}$ has the same type over $S$.
\end{proof}

\begin{rem}
    Let $P$ be the union of an acceptable chain. By \Cref{ac props}\ref{ac:stab}, the induced group embedding $G \to \Aut(P)$ is also an embedding of topological groups (where, as usual, the topology is that of pointwise convergence).
\end{rem}

\begin{lem} \label{lem:minimality}
    Let $(M_k)_{k < l}$ be an acceptable chain, and let $P = \bigcup_{k<l} M_k$. Then:
    \begin{enumerate}[label=(\roman*)]
        \item\label{S min in tp q} each $s\in S$ is a minimal element of the set of elements of $P$ with type $q$ over $A$;
        \item\label{RT min coinitial} $R \cup T$ is exactly the set of minimal points of the set $\{m \in P \mid (m^+ \cup m^-) \cap S \text{ is finite}\}$.
    \end{enumerate}
\end{lem}

\begin{proof}
    \ref{S min in tp q}: Let $m \in P$, and suppose $m < s$ for some $s \in S$. By \Cref{ac props}\ref{ac:s-} there is $C \fin V_p$ with $m^- \cap V_p = C^- \cap V_p$. As $p$ is an upper limit of $\lambda(A)$, by \Cref{not upper limit not lower limit} we cannot have $V_p \sub C^-$. So $m^- \cap V_p \neq V_p$, and so $m$ does not have type $q$ over $A$.
    
    \ref{RT min coinitial}: It is straightforward to check from the definition of $M_0$ that $R \cup T \sub \{m \in P \mid (m^+ \cup m^-) \cap S \text{ is finite}\}$. As $R \cup T$ is an antichain, it suffices to show that $\{m \in P \mid (m^+ \cup m^-) \cap S \text{ is finite}\} \sub (R \cup T)^+$. Let $m \in P$ with $(m^+ \cup m^-) \cap S$ finite. By \Cref{ac props}\ref{ac:notrt+}, if $m \notin (R \cup T)^+$ then $m^+ \cap S \in \Sigma' \cup \{S\}$, so as $\Sigma'$ is a collection of moieties of $S$ and $S$ is infinite, we have that $m^+ \cap S$ is infinite, contradiction. 
\end{proof}

\subsection{Guaranteeing witnesses for one-point extension} \label{ss: o-p ext}

To construct an acceptable chain whose union is the generic poset, we will show that it is possible to satisfy the one-point extension property: that is, we will show that in the construction we may ensure there is a witness for each finite valid triple.

\begin{lem} \label{enough aps}
    Let $(M_k)_{k < l}$ be an acceptable chain. Let $k < l$ and let $(U_0,V_0,W_0)$ be a finite valid triple in $M_k$. 
    \begin{enumerate}[label=(\roman*)]
        \item \label{enough aps Sigma} If $U_0^- \cap S \neq \emp$, then there exists $Z \in \Sigma$ such that $(U_0 \cup Z, W_0)$ is an acceptable pair in $M_k$ with $\tau(U_0 \cup Z, W_0)$ in the open set $\langle U_0, V_0, W_0 \rangle$.
        \item \label{enough aps Sigma'} If $U_0^- \cap S = \emp$, then there exists $Z \in \Sigma' \cup \{\emp\}$ such that $(U_0, W_0 \cup Z)$ is an acceptable pair in $M_k$ with $\tau(U_0, W_0 \cup Z)$ in the open set $\langle U_0, V_0, W_0 \rangle$.
    \end{enumerate}
\end{lem}
\begin{proof}
    \ref{enough aps Sigma}: As $U_0^- \cap S \neq \emp$, we have $W_0 \sub S^+ \setminus S$ and $U_0^- \cap (R \cup T) \neq \emp$. Let
    \begin{alignat*}{2}
        \mc{U} &= \{u^- \cap S \mid u \in U_0 \cap (S^+ \setminus S)\},\quad C &&= U_0 \cap S, \\
        \mc{V} &= \{v^+ \cap S \mid v \in V_0 \setminus (R \cup T)^+\},\quad D &&= \textstyle\bigcup\{v^+ \cap S \mid v \in V_0 \cap (R \cup T)^+\}, \\
        \mc{W} &= \{w^- \cap S \mid w \in W_0\}.
    \end{alignat*}
    Note that by definition we have $C \cup \bigcup \mc{U} = U_0^- \cap S$ and $D \cup \bigcup \mc{V} = V_0^+ \cap S$.
    
    Then:
    \begin{itemize}
        \item by \Cref{ac props}\ref{ac:nots} we have $\mc{U} \sub \Sigma$ and $\mc{W} \sub \Sigma$;
        \item by \Cref{ac props}\ref{ac:notrt+} we have $\mc{V} \sub \Sigma'$ (note that as $(U_0, V_0, W_0)$ is a valid triple and $U_0^- \cap S \neq \emp$, there is no $v \in V_0$ with $v^+ \cap S = S$);
        \item by the definition of $M_0$ (specifically, that all pairs of elements of $S$ and $T$ are incomparable and also that each element of $R$ is less than only finitely many elements of $S$), we have that $D$ is finite;
        \item as $(U_0, V_0, W_0)$ is a valid triple we have $C \cup \bigcup \mc{U} \sub \bigcap \mc{W}$ and $(C \cup \bigcup \mc{U}) \cap (D \cup \bigcup \mc{V}) = \emp$;
        \item as $W_0 \sub S^+$ and $U_0 < W_0$, by \Cref{ac props}\ref{ac:s+orb},\ref{ac:s+diff} we have $\mc{U} \cap \mc{W} = \emp$.
    \end{itemize}
    So by \Cref{build Sigmas}\ref{ZinSigma} there are infinitely many $Z \in \Sigma$ with $C \cup \bigcup \mc{U} \sub Z \sub \bigcap \mc{W}$ and $Z \cap V_0^+ = Z \cap (D \cup \bigcup \mc{V}) = \emp$. We take one such $Z$ with the additional property that $Z$ does not lie in the set $\{m^- \cap S \mid m \in M_k \setminus S\}$ (this set is finite by \Cref{ac props}($\ast$)). We then have $U_0^- \cap S \sub Z$ and $Z < W_0$, and as $Z \cap V_0^+ = \emp$ we also have $V_0 \cap Z^- = \emp$. So $(U_0 \cup Z, W_0)$ is an acceptable pair with $\tau(U_0 \cup Z, W_0) \in \langle U_0, V_0, W_0 \rangle$. 

    \ref{enough aps Sigma'}: If $U_0^- \cap (R \cup T) \neq \emp$, then trivially $(U_0, W_0)$ is an acceptable pair and $\tau(U_0, W_0) \in \langle U_0, V_0, W_0 \rangle$. So assume $U_0 ^- \cap (R \cup T) = \emp$. If $W_0^+ \cap S = S$ then we immediately have that $(U_0, W_0)$ is an acceptable pair, so we may assume $W_0^+ \cap S \neq S$. Let
    \begin{alignat*}{2}
        \mc{U} &= \{u^+ \cap S \mid u \in U_0, u^+ \cap S \neq S\},\\
        \mc{V} &= \{v^- \cap S \mid v \in V_0 \cap (S^+ \setminus S)\},\quad &&D = V_0 \cap S,\\
        \mc{W} &= \{w^+ \cap S \mid w \in W_0 \setminus (R \cup T)^+\},\quad &&C = \textstyle\bigcup \{w^+ \cap S \mid w \in W_0 \cap (R \cup T)^+\}.
    \end{alignat*}
    Note that by definition we have $C \cup \bigcup \mc{W} = W_0^+ \cap S$ and $D \cup \bigcup \mc{V} = V_0^- \cap S$.
    
    Then:
    \begin{itemize}
        \item by \Cref{ac props}\ref{ac:notrt+} we have $\mc{U} \sub \Sigma'$ and $\mc{W} \sub \Sigma'$;
        \item by \Cref{ac props}\ref{ac:nots} we have $\mc{V} \sub \Sigma$;
        \item $C$ is finite by an analogous argument to that for $D$ in case \ref{enough aps Sigma};
        \item as $(U_0, V_0, W_0)$ is a valid triple we have $C \cup \bigcup \mc{W} \sub \bigcap \mc{U}$ and $(C \cup \bigcup \mc{W}) \cap (D \cup \bigcup \mc{V}) = \emp$.
    \end{itemize}

    So by \Cref{build Sigmas}\ref{ZinSigma'} there is $Z \in \Sigma'$ with $W_0^+ \cap S = C \cup \bigcup \mc{W} \sub Z \sub \bigcap \mc{U}$ and $Z \cap V_0^- = Z \cap (D \cup \bigcup \mc{V}) = \emp$. Recalling that $U_0 \cap S = \emp$, we thus have $U_0 < W_0 \cup Z$. So $(U_0, W_0 \cup Z)$ is an acceptable pair and $\tau(U_0, W_0 \cup Z) \in \langle U_0, V_0, W_0 \rangle$.
\end{proof}

\subsection{The stabiliser lemma} \label{ss: stabilisers}

The previous section, together with a standard bookkeeping argument to schedule the tasks of adding witnesses for each finite valid triple (as in the usual proof of \Frn's theorem -- see \cite[Section 7.1]{Hod93}), shows that there exists an acceptable chain $(M_k)_{k < \omega}$ whose union is the generic poset. To guarantee unique extension of automorphisms of $A$, we further refine the construction by controlling better the action of $G$ on $S^+$.

\begin{notn}
    Let $(M_k)_{k < l}$ be an acceptable chain, and let $P = \bigcup_{k < l} M_k$. When we write $S^+$, this means the upward-closure of $S$ in $P$. For $k < l$, we define $S^+_k = S^+ \cap M_k$.
\end{notn}

\begin{lem}\label{enough aps2}
    Let $(M_k)_{k < l}$ be an acceptable chain with $l > 0$ finite. Let $(U_0, V_0, W_0)$ be a finite valid triple in $M_{l-1}$. Assume $U_0^- \cap S \neq \emp$.
    
    Then we can extend $(M_k)_{k < l}$ to an acceptable chain $(M_k)_{k < l'}$ such that $M_{l'-1}$ contains a point with a type lying in the open set $\langle U_0, V_0, W_0 \rangle$, and the following holds:
    \[ \text{for all } l \leq k < l' \text{ and } m \in M_k \setminus M_{k - 1}, \text{ we have } \St(m) = \St(\tp(m / M_0 \,\cup\, S^+_{k - 1})). \tag{$\star$}\]
\end{lem}
\begin{proof}
    First observe that, given an acceptable chain up to stage $M_k$, to prove ($\star$) for $k$ it suffices to show the equality for a single point of $M_k \setminus M_{k - 1}$, as $M_k \setminus M_{k - 1}$ forms a $G$-orbit and the $G$-action on this orbit is determined by the action on types over $M_{k-1}$ (see \Cref{def:acceptable-pair}). Also observe that given $G$-posets $Q' \sub Q \sub P$, for $m \in P$ we have $\St({\tp(m/Q)}) \sub \St({\tp(m/Q')})$. We will use these observations repeatedly throughout this proof.

    Recall that $A = V_p \cup W_p$. For $a\in V_p$ we define $t'_a = t^{}_a$, and for $a \in W_p$ we define $t'_a = a$. Note that $t'_a \notin S^-$ for all $a \in A$, and also note that as $p$ is $G$-fixed, for all $g \in G$ we have $gV_p = V_p$, $gW_p = W_p$. Let $s_0 \in S$, and for $a \in A$ let $T'_a = \{t'_a, s^{}_0\}$. Also, by \Cref{ac props}\ref{ac:stab} there is $A_0 \fin A$ with $\St({A_0}) \sub \St({U_0 \cup W_0})$ (apply \Cref{ac props}\ref{ac:stab} to each element of $U_0 \cup W_0$ and then take the union of the finite sets obtained). Enumerate $A_0=\{a_i \mid i < d\}$.
    
    We define $M_{l + i}$ for $i < d$ inductively, together with associated distinguished points $e_i \in M_{l + i}$. Let $i < d$, and assume we have defined $M_{l + j}$ and $e_j \in M_{l + j}$ for $j < i$. By \Cref{enough aps}\ref{enough aps Sigma}, there is $Z_i \in \Sigma$ such that $(T'_{a_i} \cup Z_i, \emp)$ is an acceptable pair in $M_{l + i - 1}$. Let $M_{l + i} = M_{l + i - 1} \ast (T'_{a_i} \cup Z_i, \emp)$, and let $\tau_i = \tau_{M_{l+i-1}}(T'_{a_i} \cup Z_i, \emp)$ and $e_i = e_{\tau_i} \in M_{l + i}$. So we have $\tau_i = ((T'_{a_i} \cup Z_i)^-, M_{l+i-1} \setminus (T'_{a_i} \cup Z_i)^-, \emp)$, where the downward-closure is taken in $M_{l+i-1}$. This completes the inductive definition. Note that for each $i < d$ and for each $g \in G$ we have that $ge_i \in S^+$ and that $ge_i$ is maximal in $M_{l + i}$. We therefore also have that $\{ge_i \mid i < d, g \in G\}$ is an antichain: by \Cref{ac props}\ref{ac:s+orb} each $G$-orbit is an antichain, and given distinct $ge_i$, $g'e_{i'}$ with $i < i'$, we see that by maximality of $g'e_{i'}$ we have $ge_i \ic g'e_{i'}$ or $ge_i < g'e_{i'}$, and the latter cannot occur by the definition of $\tau_{i'}$.
    
    \textbf{Claim: $\St({\tau_i}) = \St({a_i})$ for all $i < d$, and ($\star$) holds for all $l \leq k < l + d$.}
    \begin{subproof}
        First note that it suffices to show ($\star$) for $m = e_i$, $i < d$ (by the observations at the start of the proof), and also note that $\St({e_i}) = \St({\tau_i})$. It is immediate by the definition of $\tau_i$ and the fact that $G$ pointwise-fixes $S$ that $\St({a_i}) \sub \St({\tau_i})$. As \[\St({\tau_i}) \sub \St({\tp(e_i / M_0 \cup S^+_{l + i - 1})}) \sub \St({\tp(e_i / M_0)}),\] it suffices to show $\St({\tp(e_i / M_0)}) \sub \St({a_i})$. For $g \in \St({\tp(e_i / M_0)})$ we have $t'_{ga_i} = gt'_{a_i} \in (T'_{a_i} \cup Z_i)^-$. As $t'_{ga_i} \notin S^-$, this implies $t'_{ga_i} \leq t'_{a_i}$. Similarly $t'_{a_i} \leq t'_{ga_i}$, so $t'_{a_i} = t'_{ga_i}$ and thus $ga_i = a_i$. 
    \end{subproof}

    Recall that we assumed $U_0^- \cap S \neq \emp$. By \Cref{enough aps}\ref{enough aps Sigma} there is $Z_d \in \Sigma$ such that $(U_0\cup \{e_i \mid i < d\} \cup Z_d, \emp)$ is an acceptable pair in $M_{l + d - 1}$. Let $M_{l + d} = M_{l + d - 1} \ast (U_0\cup \{e_i \mid i < d\} \cup Z_d, \emp)$, and let $\tau_d = \tau_{M_{l + d - 1}}(U_0\cup \{e_i \mid i < d\} \cup Z_d, \emp)$ and $e_d = e_{\tau_d} \in M_{l + d}$. Note that $e_d \in S^+$ and $e_d$ is maximal in $M_{l + d}$.

    \textbf{Claim: $\St({\tau_d}) = \St({A_0})$ and ($\star$) holds for $k = l + d$.}
    \begin{subproof}
        Let $g \in \St({A_0})$. Then we have that $g$ fixes every element of $U_0 \cup \{e_i \mid i < d\}\cup Z_d$ pointwise, by the following: we took $A_0$ with $\St({A_0}) \sub \St({U_0 \cup W_0})$, we have the previous claim and also we have $Z_d \sub S$. So $g\tau_d = \tau_d$.

        Let $g \in \St({\tau_d}) \sub \St({\tp(e_d / M_0 \cup S^+_{l+d-1})}) \sub \St({\tp(e_d / (M_{l+d-1} \setminus M_{l-1}))})$. (Here, note that $M_{l+d-1} \setminus M_{l-1} = G \cdot \{e_j \mid j < d\} \sub S^+_{l+d-1}$.) Then for each $i < d$ we have $ge_i \in (U_0 \cup \{e_j \mid j < d\} \cup Z_d)^-$, and as $ge_i \notin (U_0 \cup Z_d)^-$ by maximality and $G \cdot \{e_j \mid j < d\}$ is an antichain, we must have $ge_i = e_i$. So $g \in \bigcap_{i < d} \St({e_i}) = \St({A_0})$.
    \end{subproof}
    
    As $e_d$ is maximal in $M_{l + d}$ and $e_d > U_0$, we have that $(U_0, V_0, \{e_d\} \cup W_0)$ is a valid triple in $M_{l + d}$. By Lemma \ref{enough aps}\ref{enough aps Sigma} there is $Z_{d + 1} \in \Sigma$ such that $(U_0 \cup Z_{d + 1}, \{e_d\} \cup W_0)$ is an acceptable pair in $M_{l + d}$ and $\tau_{d + 1} := \tau_{M_{l + d}}(U_0 \cup Z_{d + 1}, \{e_d\} \cup W_0) \in \langle U_0, V_0, W_0 \rangle$. Let $M_{l + d + 1} = M_{l + d} \ast (U_0 \cup Z_{d + 1}, \{e_d\} \cup W_0)$ and $e_{d + 1} = e_{\tau_{d + 1}}$. Note that $e_{d + 1} \in S^+$.
    
    \textbf{Claim: $\St({\tau_{d + 1}}) = \St({A_0})$ and ($\star$) holds for $k = l + d + 1$.}
    \begin{subproof}
        Let $g \in \St({A_0})$. Then $g$ fixes every element of $U_0 \cup Z_{d + 1} \cup \{e_d\} \cup W_0$ pointwise, so $g\tau_{d + 1} = \tau_{d + 1}$.

        Let $g \in \St({\tau_{d + 1}}) \sub \St({\tp(e_{d + 1} / M_0 \cup S^+_{l + d})}) \sub \St({\tp(e_{d + 1} / (M_{l + d} \setminus M_{l + d - 1}))})$. Then $ge_d \in (\{e_d\} \cup W_0)^+$. Suppose for a contradiction that $ge_d^{} \in W_0^+$. Then $ge_d > w$ for some $w \in W_0$, so $e_d > g^{-1}w$ and thus $g^{-1}w \in (U_0\cup \{e_i \mid i < d\} \cup Z_d)^-$. Recalling that $U_0^- \cap S \neq \emp$ and $U_0 < W_0$ and $S$ is an antichain, we have $w \in S^+ \setminus S$, so $g^{-1}w \in S^+ \setminus S$. As $Z_d \sub S$, we have $g^{-1}w \notin Z_d^-$. As $U_0 < W_0$, if $g^{-1}w \in U_0^-$ then $g^{-1}w < w$, contradicting \Cref{ac props}\ref{ac:s+orb}. If $g^{-1}w \in \{e_i \mid i < d\}^-$, then $g^{-1}w < e_i$ for some $i < d$, and so $g^{-1}w \in (T'_{a_i} \cup Z_i)^-$. But $g^{-1}w \in S^+ \setminus S$, contradiction. 
        
        So $ge_d^{} \in e_d^+$. As $e_d$ is maximal in $M_{l + d}$ we have $ge_d = e_d$. By the previous claim we have $\St({e_d}) = \St({A_0})$, so $g \in \St({A_0})$.
    \end{subproof}
    We therefore see that $(M_k)_{k < l + d + 2}$ has the required properties: we have $\tp(e_{d + 1} / M_{l + d}) = \tau_{d + 1} \in \langle U_0, V_0, W_0 \rangle$, and $(\star)$ holds for $l \leq k < l + d + 2$. 
\end{proof}

\subsection{The main theorem}

We now conclude the proof of our main theorem.

\begin{proof}[Proof of \Cref{main thm}]
    By \Cref{enough aps} and \Cref{enough aps2}, there is an acceptable chain $(M_k)_{k < \omega}$ satisfying the following properties:
    \begin{enumerate}[label=(\alph*)]
        \item\label{fin triples witnessed} for each $k < \omega$ and for every finite valid triple $(U_0, V_0, W_0)$ in $M_k$, there is $l > k$ and a point $m \in M_l$ with $\tp(m / M_{l-1}) \in \langle U_0, V_0, W_0 \rangle$;
        \item\label{nice stabilisers} for each $1 \leq k < \omega$ and $m \in S^+_{k} \setminus S^+_{k - 1}$, we have $\St({m}) = \St({\tp(m / M_0 \,\cup\, S^+_{k - 1})})$;
        \item\label{fin aps not S+ witnessed} for each $k < \omega$ and acceptable pair $(U_0, W_0)$ of finite subsets of $M_k$ with $U_0^- \cap S = \emp$, there is $l > k$ and $m \in M_l$ with $\tp(m/M_k) = \tau_{M_k}(U_0, W_0)$.
    \end{enumerate}
    
    (Here, there are two types of task to be completed: finding witnesses for each finite valid triple so that conditions \ref{fin triples witnessed} and \ref{nice stabilisers} hold, and finding witnesses for each finite acceptable pair to satisfy condition \ref{fin aps not S+ witnessed}. It is immediate how to complete the latter task and satisfy condition \ref{fin aps not S+ witnessed}, using the definition of an acceptable extension. To satisfy conditions \ref{fin triples witnessed} and \ref{nice stabilisers}: given a finite valid triple $(U_0, V_0, W_0)$, if $U_0^- \cap S = \emp$ then we use \Cref{enough aps}\ref{enough aps Sigma'}, and if $U_0^- \cap S \neq \emp$ then we use \Cref{enough aps2}. We use a standard bookkeeping argument for the tasks: that is, we schedule the countably many tasks to be completed with a scheduling function similar to that used in the usual proof of \Frn's theorem, found in \cite[Section 7.1]{Hod93} -- we omit the details.)

    Let $P = \bigcup_{k < \omega} M_k$. Condition \ref{fin triples witnessed} implies that $P$ is (isomorphic to) the generic poset. By construction, $P$ is a $G$-poset extending $G \curvearrowright A$, so every automorphism of $A$ extends to an automorphism of $P$. It remains to show uniqueness of extension.
    
    Let $h \in \Aut(P)$ with $h|_A = \id_A$. We will show that $h = \id_P$. We do this via the following steps:
    \begin{enumerate}[label=(\roman*)]
        \item\label{h fixes S sw} $h$ fixes $S$ setwise;
        \item\label{h fixes R sw and T sw} $h$ fixes $R$ setwise and fixes $T$ setwise;
        \item\label{h fixes RTS pw} $h$ fixes $R \cup T \cup S$ pointwise;
        \item\label{h fixes S+ pw} $h$ fixes $S^+$ pointwise.
    \end{enumerate}
    The final claim gives that $h = \id_P$, via \Cref{uniq on +-closed}. We now prove these claims in order.
    
    \ref{h fixes S sw}: By \Cref{lem:minimality}\ref{S min in tp q}, each point in $S$ is minimal in the set of points in $P$ with type $q$ over $A$. We now show that every point that is minimal in the set of points in $P$ with type $q$ over $A$ is an element of $S$. Let $m \in P$ be such a point and suppose for a contradiction that $m \notin S$. Then $m \notin S^- \cup S^+$ by minimality of each point of $S$ and by the minimality of $m$. So $m \ic S$. Let $k$ be such that $m \in M_k \setminus M_{k-1}$ and let $(U,W)$ be an acceptable pair in $M_{k-1}$ with $\tp(m / M_{k - 1}) = \tau(U,W)$. As $m \ic S$ we have $U \cap S = W \cap S = \emp$, so $U$ and $W$ are finite. Also $m \ic S$ implies $U^- \cap S = \emp$, and in addition $m^+ \cap S = W^+ \cap S$, so $(U, W \cup \{m\})$ is an acceptable pair in $M_k$. Thus by condition \ref{fin aps not S+ witnessed} of our acceptable chain there is $n \in P \setminus M_k$ satisfying $\tp(n / M_k) = \tau_{M_k}(U, W \cup \{m\})$. Then $n^- \cap A = U^- \cap A = m^- \cap A$ and $n^+ \cap A = (W^+ \cap A) \cup (m^+ \cap A) = m^+ \cap A$. So $\tp(n/A) = q$, but $n<m$, so $m$ is not minimal, contradiction.

    So $S$ is exactly the set of minimal points of the set of points in $P$ with type $q$ over $A$. As the set of points in $P$ with type $q$ over $A$ is fixed setwise by $h$, we have that $S$ is fixed setwise by $h$.
    
    \ref{h fixes R sw and T sw}: We have that $h$ fixes $R \cup T$ setwise immediately by the previous claim and \Cref{lem:minimality}\ref{RT min coinitial}. As each element of $R$ is related to some element of $S$ and each element of $T$ is not related to any element of $S$, we have that $h$ fixes $R$ setwise and fixes $T$ setwise.
    
    \ref{h fixes RTS pw}: Each point in $T$ has a different type over $A$, so $h|_T = \id_T$. As $h$ fixes $R \cup S$ as a set and the partial order on $R \cup S$ is rigid, we have $h|_{R \cup S} = \id_{R \cup S}$.
    
    \ref{h fixes S+ pw}: We show by induction that $h$ fixes $S^+_{k}$ pointwise for $k < \omega$. As $S^+_{0} = S$, the case $k = 0$ is given by the previous claim. Suppose that $h$ fixes $S^+_{k - 1}$ pointwise. Let $m \in S^+_{k} \setminus S^+_{k - 1}$. As $\tp(m/S) = \tp(h(m)/S)$, in particular we have $m^- \cap S = h(m)^- \cap S$, and so by \Cref{ac props}\ref{ac:s+diff} we have that $h(m) = gm$ for some $g \in G$. By the induction assumption and the previous claim we have $h|_{M_0 \,\cup\, S^+_{k - 1}} = \id_{M_0 \,\cup\, S^+_{k - 1}}$, so $gm$ and $m$ have the same type over $M_0 \cup S^+_{k - 1}$. So $g \in \St({\tp(m / M_0 \,\cup\, S^+_{k - 1})})$, and by condition \ref{nice stabilisers} in our construction of $P$, we have $g \in \St(m)$. So $h(m) = gm = m$.
\end{proof}

\section*{Acknowledgements} We would like to thank Alessandro Codenotti and Aristotelis Panagiotopoulos for discussions in the joint project with the second and third authors on generic embeddings into \Fr structures, which were a useful tool in early attempts to prove the main result (though they do not appear in this paper). We would also like to thank Adam Barto\v{s}, David Bradley-Williams and Wies\l{}aw Kubi\'{s} for initial discussions, and we would like to thank Itay Kaplan for his talk in M\"{u}nster in September 2023 on the automorphism group of the Rado meet-tree, which indirectly sparked our interest in universality of automorphism groups. Finally, we would like to thank the anonymous referee, whose thorough review of the paper improved the presentation.

\section*{Funding}

The first and second authors were funded by the Deutsche Forschungsgemeinschaft (DFG, German Research Foundation) under Germany’s Excellence Strategy EXC 2044–390685587, Mathematics M\"{u}nster: Dynamics–Geometry–Structure and CRC 1442 Geometry: Deformations and Rigidity. The second author was additionally funded by Project 24-12591M of the Czech Science Foundation (GA\v{C}R).

\bibliographystyle{alpha}
\bibliography{references}

\end{document}